%% file: Rosendal_MSS_2016_rev.tex
\documentclass[graybox]{svmult}

\usepackage{mathptmx}       % selects Times Roman as basic font
\usepackage{helvet}         % selects Helvetica as sans-serif font
\usepackage{courier}        % selects Courier as typewriter font
\usepackage{type1cm}        % activate if the above 3 fonts are
\usepackage{makeidx}         % allows index generation
\usepackage{graphicx}        % standard LaTeX graphics tool
\usepackage{multicol}        % used for the two-column index
\usepackage[bottom]{footmisc}% places footnotes at page bottom
\usepackage{color}
\usepackage{shuffle}
\usepackage{amssymb,amsmath}
\newcommand{\ad}{\mathrm{ad}}
\newcommand{\Ad}{\mathrm{Ad}}
\newcommand{\autCC}{\mathrm{Aut}(\mathcal{C})}
\newcommand{\A}{\mathcal{A}}

\newcommand{\CC}{\mathcal{C}}%
\newcommand{\cinf}{C^{\infty}(\R^D)}%
\newcommand{\derCC}{\mathrm{Der}(\mathcal{C})}
\newcommand{\id}{\mathrm{id}}
\newcommand{\cI}{\mathcal{I}}
\newcommand{\cJ}{\mathcal{J}}
\newcommand{\tg}{\tilde {\mathfrak{g}}}
\newcommand{\g}{\mathfrak{g}}
\newcommand{\G}{\mathcal{G}}
\renewcommand{\H}{\mathcal{H}}

\renewcommand{\L}{\mathcal{L}}

\newcommand{\R}{\mathbb{R}}
\newcommand{\RA}{\R\!\langle \A \rangle}

\newcommand{\W}{\mathcal{W}}
\newcommand{\Z}{\mathbb{Z}}

\newcommand{\one}{{\, 1\!\!\! 1\,}}
\makeindex             % used for the subject indexq
\begin{document}

\title*{{Hopf algebra techniques to handle  dynamical systems and  numerical integrators}}
\author{A. Murua and J.M. Sanz-Serna}
\institute{A. Murua \at Konputazio Zientziak eta A.\ A.\  Saila, Informatika
 Fakultatea, UPV/EHU, E--20018 Donostia--San Sebasti\'{a}n,  Spain. \email{Ander.Murua@ehu.es}
\and J.M. Sanz-Serna \at Departamento de Matem\'aticas, Universidad Carlos III de Madrid, E--28911 Legan\'es (Madrid),  Spain.
 \email{jmsanzserna@gmail.com}}
\maketitle
\abstract{In a series of  papers the present authors and their coworkers
 have developed a family of algebraic techniques to solve a number of problems
  in the theory of discrete or continuous dynamical systems and to analyze numerical
  integrators. Given a specific problem, those techniques construct an abstract,
  {\em universal} version of it which is solved algebraically; then, the results are
   transferred to the original problem with the help of a suitable morphism. In earlier contributions,
    the abstract problem is formulated either in the
dual of the shuffle Hopf algebra  or in the dual of the Connes-Kreimer Hopf algebra.
In the present contribution we extend these techniques to more general Hopf algebras,
which in some cases lead to more efficient computations.}

\abstract*{In a series of  papers the present authors and their coworkers  have developed a family of algebraic techniques to solve a number of problems in the theory of discrete or continuous dynamical systems and to analyze numerical integrators. Given a specific problem, those techniques construct an abstract, {\em universal} version of it which is solved algebraically; then, the results are transferred to the original problem with the help of a suitable morphism. In earlier contributions, the abstract problem is formulated either in the
dual of the shuffle Hopf algebra  or in the dual of the Connes-Kreimer Hopf algebra. In the present contribution we extend these techniques to more general Hopf algebras, which in some cases lead to more efficient computations.}

\section{Introduction}
A series of  papers \cite{part1}, \cite{part2}, \cite{orlando}, \cite{part3}, \cite{china}, \cite{juanluis},
\cite{guirao}, \cite{alfonso}, \cite{kurusch}, \cite{words}, \cite{charter}  have developed a family of
algebraic techniques to solve a number of problems in the theory of discrete or continuous dynamical systems
and to analyze numerical integrators. Given a specific problem, those techniques construct an abstract, {\em
universal} version of it which is solved algebraically; then, the result is transferred to the original
problem with the help of a suitable morphism {$\Psi$}. The abstract problem is formulated either in the dual
of the shuffle Hopf algebra of words \cite{words} or in the dual of the Connes-Kreimer Hopf algebra of rooted
trees \cite{part2}.  {Operations with elements of the  relevant dual  are mapped by  $\Psi$ into operations
with formal series of differential operators.}
 For the shuffle Hopf algebra, the
solution of the original problem appears expressed as a so-called  {\em word series} \cite{words}. {In the
Connes-Kreimer case, the resulting series for the original problem are  \emph{B-series;} the Butcher group
(the group of characters of the Connes-Kreimer Hopf algebra) and B-series first appeared in the context of
numerical analysis of differential equations (see \cite{china} for a survey) decades before the Connes-Kreimer
Hopf algebra was introduced in the context of renormalization in quantum field theory. Duals of Hopf algebras
are useful in this setting because they provide rules for composing formal series.}

In the present contribution we extend these techniques to more general Hopf algebras, which in some cases lead
to more efficient computations (cf.\ \cite{fm}).

Problems that may be treated in this form include averaging of periodically or quasiperiodically forced differential systems \cite{part1}, \cite{part2}, \cite{orlando}, \cite{part3}, \cite{guirao}, \cite{kurusch}, construction of formal invariants of motion \cite{part2}, \cite{orlando}, \cite{part3}, \cite{juanluis} computation of normal forms \cite{juanluis}, \cite{kurusch}, calculations on central manifolds~\cite{charter}, and error analysis of splitting integrators for deterministic \cite{words} or stochastic \cite{alfonso} systems of differential equations. Of course it would be impossible to take up here each of those problems; the examples in this paper only refer to the computation of high-order averaged systems for periodically forced differential equations and to the analysis of the Strang splitting formula when applied to perturbations of integrable systems

The techniques studied here go back to a number of earlier developments, in particular, mention has to be made of Ecalle's mould calculus \cite{ecalle}, \cite{ecalle2} (see \cite{sauzin}, \cite{fm}, \cite{sp1}, \cite{sp2}, \cite{sp3} for more recent contributions), and of the algebraic theory of  integrators \cite{butcher}, \cite{anderfocm}, \cite{china},  \cite{MKW2008},\cite{MKL2011}.

An outline of this paper now follows. Section~\ref{sec:afds} reviews some well-known ideas on the reformulation of differential systems in Euclidean spaces as operator differential equations. Section~\ref{sec:an example} illustrates the algebraic approach in the series of papers mentioned at the beginning of this introduction. It does so by considering a concrete  averaging problem in \(\R^5\) and explicitly finding a high-order averaged system by first working abstractly in the group of characters of the shuffle Hopf algebra.  The complexity of the computations grows very quickly with the order of the averaged system sought and this motivates the material in Section~\ref{sec:general hopf}, where we show how to work with other Hopf algebras to increase the efficiency of the algorithms. The vector fields (derivations)  appearing in the given problem \(\mathcal{P}\) in Euclidean space are written as images by a Lie algebra homomorphism \(\Psi\) mapping a suitable graded Lie algebra \(\tg\) into the Lie algebra of derivations. From \(\tg\) we construct a graded, commutative Hopf algebra \(\H\) in such a way that an \lq abstact\rq\ version of \(\mathcal{P}\) may be formally solved in the group of characters \(\G\) of \(\H\); finally the formal solution in \(\G\) is translated into a formal solution of \(\mathcal{P}\). For the concrete averaging problem in
\(\R^5\),  we present a succession of alternative Hopf algebras that make it possible to compute approximations of increasingly higher order. The final section presents material where the ideas in the paper are suitably modified to cater for problems written in perturbation form, generalizing the notion of {\em extended word series} introduced in \cite{words} and used in \cite{juanluis}, \cite{kurusch}.

Due to space constraints,  we have not attempted to present the results in the most general conceivable
scenario. For instance, it is possible to work with differential systems defined on differentiable manifolds
rather than in  Euclidean spaces and scalars could be complex rather real.

\section{Algebraic formulation of differential systems}
\label{sec:afds}

It is well known that differential systems in $\R^D$ may be interpreted as
 differential equations that describe the evolution of suitably chosen linear  operators. This section reviews that interpretation, which plays a key role in later developments. We use the following notation. The vector space $\CC=\cinf$ consists of all smooth $\R$-valued functions on $\R^D$. Functions \(\chi\in\CC\) are sometimes called observables. With respect to the pointwise multiplication of observables, the space \(\CC\) is an associative and commutative algebra. The symbol \(\mathrm{End}(\CC)\) denotes the vector space of all linear operators \( X:\CC\rightarrow\CC\). When operators are multiplied by composition, \( (X_1 X_2)(\chi) = X_1(X_2(\chi))\), \(\mathrm{End}(\CC)\) is an associative algebra with a unit: the identity operator \( I:\chi\mapsto \chi\).

Consider the initial value problem in \( \R^D\)
\begin{equation}
\label{eq:odex}
\frac{d}{dt} x(t)=f(x(t),t), \quad x(0)=x_0,
\end{equation}
with  \( f\) smooth. For each frozen value of \( t\), the vector field \( f(\cdot,t)\) defines a first-order
linear differential operator \( F(t) \in \mathrm{End}(\CC)\) that associates with each observable \( \chi \)
the observable \( F(t) \chi\in\CC\) such that{
\[F(t)\chi(x) =  f(x,t)^T \cdot \nabla \chi(x) = \sum_{j=1}^{D}
f_j(x,t) \frac{\partial}{\partial x_j} \chi(x) \]
 for each \(x =(x_1,\ldots,x_D) \in\R^D\).}
 Actually, \( F(t)\) is  a derivation of the algebra \(\CC\), i.e.
\[ F(t)(\chi_1\chi_2) = (F(t)\chi_1)\, \chi_2+\chi_1\, (F(t)\chi_2).\]
The space $\derCC\subset \mathrm{End}(\CC)$ consisting of all derivations in \( \CC\) is a Lie algebra with respect to the commutator \( [F_1,F_2] = F_1F_2-F_2F_1\).

Assuming for the time being that for each  $x_0 \in \R^d$ the solution $x(t)$ of (\ref{eq:odex})  exists
for all $t \in \R$, we may define a one-parameter family \( X(t)\), \(t\in\R\), of elements of \(\mathrm{End}(\CC)\) as follows: for each observable $\chi\in\CC$ and  each $t \in \R$, $ X(t) \chi \in \CC$ is such that $ X(t)\chi ( x(0))=\chi(x(t) ) $ for each \( x(0) \in \R^D\). Clearly each \(X(t)\) is an automorphism of the algebra \(\CC\), i.e.
\begin{equation}
\label{eq:autprop}
X(t) (\chi_1\chi_2) = X(t) (\chi_1)\, X(t) (\chi_2),
\end{equation}
for any \(\chi_1, \chi_2\in\CC\). The set \(\autCC\) of all algebra automorphisms is a group for the composition of operators.

Since given $\chi \in \CC$,
\[
\frac{d}{dt} \chi(x(t)) = \chi'(x(t))\cdot f(x(t),t),
\]
we have that
\begin{equation}
\label{eq:odeX}
\frac{d}{dt} X(t) = X(t) F(t), \quad X(0)=I,
\end{equation}
or equivalently
\begin{equation}
\label{eq:odeXint}
X(t) = I + \int_{0}^{t} X(s) F(s) \, ds.
\end{equation}
In this way the solvability of (\ref{eq:odex}) implies the solvability of the operator initial value problem (\ref{eq:odeX}). When comparing (\ref{eq:odeX}) with  (\ref{eq:odex}) we note that (\ref{eq:odeX})
is  linear in \( X\) even when (\ref{eq:odex}) is not linear in \(x \); the multiplication  of operators \( X(t)F(t)\) in (\ref{eq:odeX}) corresponds to the composition of the maps \( t\mapsto x(t)\), \( (x,t)\mapsto f(x,t)\) that appears in
(\ref{eq:odex}).

Conversely assume that  $F:\R \to \derCC$ is such that there exists a one-parameter family \(X(t)\) of elements of
\(\autCC\) satisfying (\ref{eq:odeX}). We then define, for each \( t\), a vector field \( f(\cdot,t)\) in \(R^D\) by setting
\( f^i(x,t) = F(t) \chi^i(x)\), \( i= 1,\dots,D\), where \( \chi^i\) is the \(i\)-th coordinate function \(\chi^i(x) = x^i\) (superscripts denote components of a vector), and consider the corresponding problem (\ref{eq:odex}). Then, it is easily checked that (\ref{eq:odex}) has, for each \( x_0\), a solution \(x(t)\) defined for all real \(t\) and the \(i\)-component of \(x(t)\) may be found as
\(x^i(t) = (X(t)\chi^i)(x_0)\). We emphasize that for this construction to work it is essential that the operators
\(X(t)\) satisfy \eqref{eq:autprop}, i.e.\ they are automorphisms of \(\CC\).

We will present below algebraic frameworks where the initial
value problem (\ref{eq:odeX}) is interpreted in a broader sense, admitting solution curves $X(t)$ that
evolve in  groups of {\em formal automorphisms} rather than in \(\autCC\). Roughly speaking such formal
 automorphisms will be  {{\em formal}} series of linear maps that preserve multiplication
 of observables as in (\ref{eq:autprop}).
Even in the case where $X(t)$ does not correspond to an actual curve in $\autCC$,
 such  formal solution curves $X(t)$ may be used to derive rigorous results on the solution $x(t)$ of (\ref{eq:odex}).

\section{An example}
\label{sec:an example}

In this section we illustrate the use of Hopf algebra techniques by means of an example: the construction of high-order averaged systems for a periodic differential system in \(\R^5\).

\subsection{A highly-oscillatory differential system}

The following system of differential equations arises in the study of vibrational resonance in an energy harvesting device \cite{sanjuan}:
\begin{align*}
\frac{dx}{dt} &=  y, \\
\frac{dy}{dt} &= \frac{1}{2} x
   \left(1-x^2\right)-y  +\frac{v}{20}   +A \cos\left(\frac{t}{10}\right)+\omega^2 \cos (\omega t),  \\
\frac{dv}{dt} &=  -\frac{v}{100}-\frac{y}{2}.
\end{align*}
Here \( v \) is the voltage across the load resistor, \( x \) and \( y \) are auxiliary state variables, and \( \omega \gg 1 \) is the frequency of the environmental vibration. The aim is to investigate the effect that the value of the amplitude  \( A \) of the low-frequency forcing has on the output \( v \).

Averaging, i.e.\ reducing the time-periodic system to an autonomous system with a help of a periodic change of variables \cite{arnoldode}, \cite{SVM07}, is a very helpful tool to study this kind of problem  \cite{guirao}.  To average the vibrational resonance problem above, we begin by introducing new variables
\begin{equation}\label{eq:change}
x=X-\cos (t \omega), \quad y=Y+\omega  \sin (t \omega) + \cos (t \omega), \quad
v=V+\frac{1}{2} \cos (t\omega ),
\end{equation}
chosen to ensure that in the transformed system
\begin{align}\label{eq:XYV}
\frac{dX}{dt} &=  Y+\cos (t\omega), \\
\frac{dY}{dt} &=  -\frac{X}{4}-\frac{X^3}{2}-Y+\frac{V}{20}
+A  \cos \left({ \frac{t}{10}}\right)
\nonumber \\
 &\qquad\qquad +\left(\frac{3 X^2}{2}-\frac{11}{10}\right) \cos (t \omega ) -\frac{3}{4} X \cos (2 t \omega )+\frac{1}{8} \cos (3 t\omega),\nonumber  \\
\frac{dV}{dt} &=  - \frac{V}{100}-\frac{Y}{2}-\frac{101}{200}\cos(t\omega),\nonumber
\end{align}
the highly oscillatory terms have amplitudes of size \(\mathcal{O}(1)\) as \(\omega \rightarrow \infty\).
Suppression of the terms that oscillate with high frequency then results in the averaged system
\begin{align*}
\frac{dX}{dt} &=  Y, \\
\frac{dY}{dt} &=  -\frac{X}{4}-\frac{X^3}{2}-Y+\frac{V}{20}
+A  \cos \left({ \frac{t}{10}}\right),  \\
\frac{dV}{dt} &=  - \frac{V}{100}-\frac{Y}{2},
\end{align*}
whose solutions approximate   {the solution $(X(t),Y(t),V(t))$ of \eqref{eq:XYV} with errors of size
\(\mathcal{O}(1/\omega)\)  in  bounded intervals \(0\leq t\leq T<\infty\)}. Approximations with
\(\mathcal{O}(1/\omega)\) errors  (first-order averaging) to the original state variables \( x\), \(y\) \(
v\), are then obtained from \eqref{eq:change}. Approximations to \( x\), \(y\) \( v\), with errors
\(\mathcal{O}(1/\omega^2)\) (second-order averaging)  may be obtained by changing variables in \eqref{eq:XYV}
so as to reduce the amplitude of the highly oscillatory terms from \(\mathcal{O}(1)\) to
\(\mathcal{O}(1/\omega)\) and then discarding the highly oscillatory terms. The iteration of the procedure
leads successively to approximations with errors \(\mathcal{O}(1/\omega^n)\) for \(n = 3,4,\dots\) (high-order
averaging).

The averaged systems found in this way are nonautonomous since the low-frequency forcing is not averaged out. In order to deal with autonomous averaged problems we introduce two additional real-valued state variables \(C,S\) satisfying
\[\frac{dC}{dt} = -\frac{S}{10},\qquad   \frac{dS}{dt} = \frac{C}{10}\] and with initial conditions \(C(0)= 1\), \(S(0)=1\), so that \(C(t) = \cos(t/10)\), and write  problem \eqref{eq:XYV} as
\begin{align}\label{eq:ode5}
\frac{dX}{dt} &=  Y+\cos (t\omega), \\
\frac{dY}{dt} &=  -\frac{X}{4}-\frac{X^3}{2}-Y+\frac{V}{20}
+A C
\nonumber \\
 &\qquad\qquad +\left(\frac{3 X^2}{2}-\frac{11}{10}\right) \cos (t \omega ) -\frac{3}{4} X \cos (2 t \omega )+\frac{1}{8} \cos (3 t\omega),\nonumber  \\
\frac{dV}{dt} &=  - \frac{V}{100}-\frac{Y}{2}-\frac{101}{200}\cos(t\omega),\nonumber\\
\frac{dC}{dt} & = - \frac{S}{10},\nonumber\\
\frac{dS}{dt} & =  \frac{C}{10}.\nonumber
\end{align}
Note that this system in \(\R^5\) is of the form \eqref{eq:odex} with
\begin{equation*}
f(x,t) = f_a(x) +  \cos (t\omega)\,  f_b(x) +  \cos (2 t\omega)(x) \, f_c+  \cos (3 t\omega) \, f_d(x).
\end{equation*}
It is trivial to write down the derivations  \( F_a\), \dots,  \( F_d\), associated with \(f_a\), \dots,  \( f_d\).
For instance:
\[
F_a = Y\partial_X +  \left(-\frac{X}{4}-\frac{X^3}{2}-Y+\frac{V}{20}+AC\right)\partial_Y+
 \left(- \frac{V}{100}-\frac{Y}{2}\right)\partial_V - \frac{S}{10}\partial_C
+ \frac{C}{10} \partial_S.
\]
Then the derivation corresponding to \(f(x,t)\) is, for each \(t\),
\begin{equation}\label{eq:FFFF}
F_a +  \cos (t\omega)\,  F_b +  \cos (2 t\omega) \, F_c+  \cos (3 t\omega) \, F_d.
\end{equation}

\subsection{Solving the oscillatory problem with word series}
\label{sec:words}

We now introduce the alphabet $\A = \{a,b,c,d\}$,  the corresponding (infinite) set $\W$ of all words \( a\), \dots , \(d\), \(aa\), \(ab\), \dots, \(dd\), \(aaa\), \dots (including the empty word \(1\)) and  the free associative algebra $\RA$  consisting of all the {\em linear combinations} of words with real coefficients. Multiplication \( \star\) in $\RA$ is defined by concatenating words \cite{reu}, which implies that \(1\) is the unit of this (noncommutative) algebra.

Furthermore we consider again  the vector space of linear combinations of words but now endow it with
the (commutative) shuffle product \(\shuffle\) and denote by \(\H\) the resulting (shuffle) algebra. Actually \( \H\) is
a Hopf algebra for the deconcatenation coproduct. This  algebra is {\em graded}; its graded component of degree \(n\), \( n = 0,1,\dots\), consists of the linear combinations of words with \(n\) letters. The dual vector space \(\H^*\) may be identified with the set of all  {\em formal series} \(\alpha\) of the form \(\sum_\W c_ww\) for real \(c_w\in\R\) so that the image \(\langle \alpha, w \rangle\) of the word \(w\) by the linear form \(\alpha\) is the coefficient \(c_w\). Thus
\(\H^*\) is a much larger space than $\RA$. Note that the concatenation product  \( \star\) may be extended from $\RA$ to \(\H^*\) in an obvious way.
We denote by $\G\subset \H^* $ the group  of characters of $\H$ consisting of the elements \(\gamma\in\H^*\)  that satisfy the shuffle relations: \(\langle \gamma, w\shuffle w^\prime\rangle = \langle \gamma, w \rangle \langle \gamma, w^\prime\rangle\) for all words \(w\), \(w^\prime\). The Lie algebra of infinitesimal characters
 $\g\subset \H^*$   consists of those \(\beta \in\H^*\) such that \(\langle \beta,w\shuffle w^\prime\rangle =
\langle \beta,w\rangle   \langle 1, w^\prime\rangle  +  \langle 1,w\rangle \langle \beta, w^\prime\rangle\) for each pair of words. Characters and infinitesimal characters are related through the relations $\G = \exp(\g)$, $\g = \log(\G)$), i.e. each element $\gamma$ in the group is the exponential $1+\beta+(1/2)\beta\star \beta+\cdots$ of the element $\beta = (\gamma-1) - (1/2) (\gamma-1)\star(\gamma-1) + \cdots$
See \cite[Sec. 6.1]{words} for a review of the constructions above.

 To solve \eqref{eq:ode5}, we associate with each letter in \( \A\) the corresponding derivation in the expression \eqref{eq:FFFF}, i.e.\ we set
\begin{equation}\label{eq:Psi}
\Psi(a)=F_a, \quad \Psi(b)=F_b, \quad \Psi(c)=F_c, \quad \Psi(d)=F_d,
\end{equation}
and extend the mapping \( \Psi\)  to an algebra morphism from \( \RA\) to the algebra $\mathrm{End}_{\R}(\CC)$, $D=5$, by setting \( \Psi(aa) = F_aF_a\), \( \Psi(ab) = F_aF_b\), etc.
The free Lie algebra $\L(\A)$ is the linear subspace of $\RA$ consisting of linear combinations of iterated commutators such as \( [a,b] = ab-ba\), \( [a,[a,b]] = a[a,b]-b[a,b] = aab-aba-bab+bba\), \dots (the letters \( a\), \dots, \(b\) are seen as iterated commutators of order \(n=1\)). This Lie algebra is graded; its graded component of degree \( n\), \( n= 1, 2,\dots\), consists of the linear combinations of iterated commutators involving words with \(n\) letters.
 The restriction of \(\Psi\) to  $\L(\A)$  is a Lie algebra morphism $\L(\A) \to \mathrm{Der}_{\R}(\CC)\subset \mathrm{End}_{\R}(\CC) $. Note that, for fixed \(t\), \eqref{eq:FFFF} is the image under \(\Psi\) of the
 element
 \begin{equation}\label{eq:b}
 \beta(t) = a +  \cos (t\omega)\,  b +  \cos (2 t\omega) \, c+  \cos (3 t\omega) \, d\in\L(\A).
 \end{equation}

The \lq abstract\rq\ initial value problem
\begin{equation}\label{eq:absivp}
\frac{d}{dt} \alpha(t) =  \alpha(t) \star \beta(t), \quad
\alpha(0) =1,
\end{equation}
where at the outset \(\alpha(t)\) is sought as a curve in \( \RA\) is such that the mapping \(\Psi\) transforms it into the   operator  initial value problem \eqref{eq:odeX} corresponding to \eqref{eq:ode5}.
We shall solve  \eqref{eq:absivp}, and then the application of \(\Psi\) will lead to a solution of \eqref{eq:ode5}.

 We  recall that
 for  integrable\footnote{More precisely it is sufficient to ask that, for each word, the real-valued function \(\langle \beta(t), w \rangle\) be locally integrable.}  curves \( \beta(t)\) in \(\g\supset \L(\A) \) (and in particular for \(\beta(t)\) in \eqref{eq:b}), the problem
\eqref{eq:absivp} possesses a unique formal solution \(\alpha(t)\) that for each \(t\) lies in the space of formal series \(\H^* \supset \RA\). This solution may be found by a Picard iteration  and is given by a Chen series \cite{reu}
\begin{align*}
\alpha(t) &=\sum_{w \in \W} \langle \alpha(t), w \rangle\,  \, w,
\end{align*}
where for each  $w\in W$ the coefficient, $\langle \alpha(t), w\rangle$ has a known expression as an iterated integral (see e.g.\ \cite[Sec.\ 2.1]{words}, \cite[Sec.\ 2.1]{kurusch} for  details). Furthermore, for each \(t\), $\alpha(t)$ satisfies the shuffle relations and therefore belongs to the group of characters $\G\subset \H^* $. In other words, when seen as a nonautonomous initial value problem to determine a curve \(\alpha(t)\) in the group \(\G\) given a curve \(\beta(t)\) in the algebra \(\g\),  \eqref{eq:absivp} is uniquely solvable (see e.g.\ \cite[Sec 2.2.4]{words}). For each fixed \( t \), \( \Psi(\alpha(t)) \) (\(\Psi\) is applied in the obvious term by term way) is a formal series whose terms belong to \(\mathrm{End}(\CC)\) (they are actually differential operators). Furthermore the fact that \(\alpha(t)\in\G\) implies (see e.g.\ \cite[Sec. 6.1.3]{words}) that the formal series \( \Psi(\alpha(t)) \) satisfies \eqref{eq:autprop}, i.e.\ it is formally an automorphism, and  by proceeding as in the preceding section we then find that the solutions of our problem in \(\R^5\) may be represented as  formal series
\[
x(t) = \sum_{w\in\W} \langle \alpha(t), w \rangle f_w(x(0)),\qquad x(0) \in\R^5,
\]
where the mappings \( f_w:\R^5\rightarrow \R^5\) are the so-called {\em word basis functions} \cite{words}; the \(i\)-th component of \(f_w\) is obtained by applying to the \(i\)-coordinate function \(\chi^i\) the endomorphism \(\Psi(w)\). Series of this form are called {\em word series} \cite{words}, \cite{juanluis}, \cite{china}, \cite{kurusch}, \cite{alfonso}.

\subsection{Averaging with word series}

We now  average \eqref{eq:ode5} by first averaging its abstract version \eqref{eq:b}--\eqref{eq:absivp}.
We seek a \( 2\pi/\omega\)-periodic map $\kappa:\R \to \G$ and a (time-independent) $\bar \beta \in \g$ such that
\begin{equation}\label{eq:sanseb}
\frac{d}{dt} \kappa(t) = \kappa(t) \star \beta(t) - \bar \beta \star \kappa(t).
\end{equation}
It is easily checked  that then $\alpha(t) = \exp(\bar \beta\, t) \star \kappa(t)$; in this way the formal solution \(\alpha(t)\)  of the periodic problem \eqref{eq:b}--\eqref{eq:absivp} is obtained, via the {\em periodic map \(\kappa(t)\)}, from the solution \(\bar \alpha(t) =\exp(\bar \beta\, t) \) of the linear {\em autonomous} problem \( (d/dt) \bar \alpha = \bar\alpha(t)\star \bar\beta\), \(\bar\alpha(0) = 1\) (the averaged problem).

There is some freedom when solving \eqref{eq:sanseb}. In {\em stroboscopic averaging} one imposes the additional condition
 $\kappa(0)=1$,  so that the averaged solution \(\bar\alpha(t)\)  coincides with \(\alpha(t)\) at all stroboscopic times \(t_k=k(2\pi /\omega)\), \( k\in \Z\) \cite{part2}. Alternatively, it is also possible to impose the {\em zero-mean} condition
\begin{equation}\label{eq:zeromean}
\int_0^{2\pi/\omega} \log(\kappa(t)) dt=0.
\end{equation}
(Note that the stroboscopic condition demands that \(\log(\kappa(t))\) vanishes at \(t=0\) rather than on average over a period as in \eqref{eq:zeromean}.)

By proceeding recursively with respect to the number of letters in the words involved, the stroboscopic condition (respectively the zero-mean condition) and \eqref{eq:sanseb} uniquely determine all the coefficients of the formal series \(\bar \beta\) and \(\kappa(t)\).\footnote{For stroboscopic averaging, the recursions that allow the simple computation of the coefficients of
\(\bar \beta\) and \(\kappa(t)\) may be seen in \cite{part2} or \cite{kurusch}, but in those references \(\bar\beta\) and \(\kappa(t)\) are found with the help of an auxiliary transport equation rather than via \eqref{eq:sanseb}.} We have implemented the corresponding recursions in a symbolic manipulation package. As an example, when truncating the series for \(\bar \beta\) so as to only keep  words with three or less letters, we find, in the zero-mean case:
\begin{align*}
\bar \beta^{[3]} &={a} + \frac{1}{\omega^2} \left(
\textstyle \frac{1}{4}\, {abb} -\frac{1}{2} \, {bab} +\frac{1}{4} \, {bba}
+\frac{1}{16}\, {acc}-\frac{1}{8} \, {cac} +\frac{1}{16} \, {cca} \right. \\
& \, \textstyle +\frac{1}{36} \, {add} -\frac{1}{18} \, {dad} +\frac{1}{36} \, {dda}
-\frac{1}{8} \, {bbc}  +\frac{1}{4} \, {bcb} -\frac{1}{8} \, {cbb} \\
&\, \textstyle  \left.
-\frac{1}{12} \, {bcd}+\frac{1}{8} \, {bdc}-\frac{1}{24} \, {cbd}
+\frac{1}{8} \, {cdb} -\frac{1}{24} \, {dbc} -\frac{1}{12} \, {dcb}
\right) \in \g.
\end{align*}
The corresponding result under the stroboscopic condition is similar but includes more terms (40 rather than 19).

Now that the problem \eqref{eq:b}--\eqref{eq:absivp} has been averaged, we apply the transformation \(\Psi\) to average our problem in the Euclidean space \(\R^5\). From \(\bar \beta\) we obtain the formal vector field given by the word-series
\[
\bar f(x) = \sum_{w\in\W} \langle \bar\beta, w \rangle f_w(x),
\]
and from \(\kappa(t)\) we construct the formal periodic change of variables given by the word-series
\[
U(x,t) = \sum_{w\in\W} \langle \kappa(t), w \rangle f_w(x),
\]
such that the solutions \(x(t)\) of \eqref{eq:ode5} are formally given as \( x(t) = U(\bar x(t),t)\) with
\( (d/dt) \bar x = \bar f(\bar x)\).

To deal with {\em bona fide} vector fields and changes of variables, one has to truncate the corresponding formal series.
In our example,
the truncation $\bar \beta^{[3]}$ found above leads to a vector field in \(\R^5\) which  after eliminating the auxiliary variables $C$ and $S$, reduces to the following time-dependent vector field in $\R^3$:
\begin{eqnarray*}
&& Y \, \partial_{X} +
\left( -\frac{X}{4}-\frac{X^3}{2}-Y+\frac{V}{20}+A  \,  \cos( \frac{t}{10})\right) \, \partial_{Y}
  - \left ( \frac{V}{100}+\frac{Y}{2} \right) \, \partial_{V} \\
&&  + \frac{1}{\omega^2} \left( \frac{3 X}{4}  \, \partial_X + \left(-\frac{9 X^3}{4}+\frac{51 X}{640}-\frac{3 Y}{4}\right) \,  \partial_Y  -\frac{3 X }{8} \, \partial_V \right).
\end{eqnarray*}
With the help of a truncated change of variables, the solutions of the corresponding differential system provides \(\mathcal{O}(1/\omega^3)\) approximations to
   \(X(t)\) , \(Y(t)\), \(V(t)\) in \eqref{eq:ode5}. Truncations of this kind and their accuracy are discussed in detail in \cite{orlando} and \cite{part3}.

It is important to emphasize that the construction above is {\em universal:} \(\bar \beta\) and \(\kappa(t)\) would not change if the expressions for the vector fields \(f_a\), \dots, \(f_d\) in \(\R^5\) considered above were replaced by another set of four vector fields in \(\R^D\) with arbitrary \(D\). There is a price to be paid for this generality: in our case there are \(4^n\) words with \( n\) letters and accordingly the complexity of the computations grows very quickly as \(n\) increases. In a laptop computer our computations had to be limited to \(n\leq 8\). In what follows we show how to replace the shuffle Hopf algebra \(\H\) by alternative Hopf algebras which may lead to simpler computations.

\section{General Hopf algebras}
\label{sec:general hopf}

In preceding section  we studied the operator initial value problem  \eqref{eq:odeX} with the help of a
mapping \(\Psi\) whose restriction to
 the free Lie algebra \(\L(\A)\)  is a Lie algebra morphism into the
 algebra of derivations \(\derCC\). We now study the more general situation where \(\L(\A)\)
 is replaced by a graded Lie algebra
\begin{equation}
\label{eq:gradedLA}
\tg
=\bigoplus_{n\geq 1} \g_n,
\end{equation}
with finite-dimensional homogeneous subspaces $\g_n$,\footnote{{It is not essential to assume   that each
$\g_n$  is finite dimesional. The arguments below may be readily adapted to cover the general case under the
proviso that the summation in \eqref{eq:delta} is well defined (cf. second paragraph after (24)).}}
and there are a Lie algebra
homomorphism \( \Psi: \tg \to \derCC \)
and  a curve $\beta:\R \to \tg$ such that
$\Psi(\beta(t)) = F(t)$ for all $t \in \R$.

Note that \(\Psi\) can be uniquely extended to an associative algebra homomorphism from the universal enveloping algebra $U(\tg)$ of \(\tg\) to $\mathrm{End}(\CC)$, which we denote with the same symbol $\Psi$. We shall use the symbol $\star$
to denote the (associative)   product in $U(\tg)$ such that $[G_1,G_2] = G_1 \star G_2- G_2 \star G_1$ for all \(G_1,G_2\in U(\tg)\). In the particular case where \eqref{eq:gradedLA} is the free Lie algebra generated by a finite alphabet \(\A\), $U(\tg)$ coincides with \(\RA\) and \(\star\) is the concatenation product.

\subsection{Solving the operator initial value problem}
\label{ss:soivp}

We denote by
\begin{equation}
\label{eq:Gi}
\{G_i \ : i \in \cI \}
\end{equation}
 a homogeneous basis of the graded Lie algebra (\ref{eq:gradedLA}), where $\cI$ is some set of indices, $\cI = \bigcup_{n\geq 1} \cI_n$, and $\{G_i \ : i \in \cI_n \}$ is a basis of $\g_n$ for each $n\geq 1$. If \(\beta(t) = \sum_{i\in\cI} \lambda_i(t)G_i\), we rewrite (\ref{eq:odeXint}) as
\[
X(t) = I +\sum_{i\in\cI} \int_0^t \lambda_i(t) X(t)\Psi(G_i)\,dt,
\]
an equation that may be solved by the following Picard iteration,
\begin{eqnarray*}
X^{[0]}(t) &=& I\\
X^{[1]}(t) &=& I +\sum_{i\in\cI} \int_0^t \lambda_i(t) X^{[0]}(t)\Psi(G_i)\,dt = I + \sum_{i\in\cI}  \left(\int_0^t \lambda_i(t)\, dt\right) \Psi(G_i),\\
X^{[2]}(t) &=& I +\sum_{i\in\cI} \int_0^t \lambda_i(t) X^{[1]}(t)\Psi(G_i)\,dt\\
\cdots & = & \cdots
\end{eqnarray*}
In this way, one may construct a formal solution $X(t)$ of (\ref{eq:odeX})
of the form
\begin{equation}\label{eq:solpicard}
X(t) = I+\sum_{m\geq 1} \sum_{(i_1,\ldots,i_m) \in \cI^m} a_{i_1,\ldots,i_m}(t) \, \Psi(G_{i_1}) \cdots \Psi(G_{i_m}).
\end{equation}

Unfortunately, this series is unnecessarily complicated as there are many linear dependencies among the endomorphisms of the form \( \Psi(G_{i_1}) \cdots\)\( \Psi(G_{i_m})\). For instance, \(\Psi(G_{i_1})\Psi(G_{i_2}) - \Psi(G_{i_2})\Psi(G_{i_1})\) has to coincide with \(\Psi([G_{i_1},G_{i_2}])\) and therefore must be a linear combination of endomorphisms \(\Psi(G_i)\), \(i\in\cI\).\footnote{In the case where \eqref{eq:gradedLA} is the free Lie algebra generated by a finite alphabet \(\A\), we saw that it is possible to write the solution \(X(t)\) as a series constructed from endomorphisms of the form \( \Psi(G_{a_1}) \cdots\)\( \Psi(G_{a_m})\), with the \(a_i\in\A\); this is far more compact than \eqref{eq:solpicard}, which involves terms \( \Psi(G_{i_1}) \cdots\)\( \Psi(G_{i_m})\) made of arbitrary elements \(G_i\) of the basis.}

 If \(<\) denotes a  total order relation in $\cI$, the
Poincar\'e-Birkhoff-Witt  (PBW) theorem ensures  that the products
\begin{equation}
\label{eq:PWBbasis}
\{G_{i_1} \star \cdots \star G_{i_m}\ : \  i_1 \leq \cdots \leq i_m\}
\end{equation}
 (including the empty product equal to the unit element $\one$) provide  a basis of  $U(\tg)$.
Therefore, it is possible to simplify \eqref{eq:solpicard} by removing the linear dependencies in the right-hand side  so as to  end up with a formal series that only uses endomorphisms of the form
$\Psi(G_{i_1}) \cdots \Psi(G_{i_m})$ with $i_1 \leq \cdots \leq i_m$.
However the basis of $U(\tg)$ given by the PBW theorem may not be the most convenient in practice\footnote{This was illustrated in the preceding section, where we used the basis of \(\RA\) consisting of words.} and in what follows we shall work with an arbitrary homogeneous basis of $U(\tg)$
\begin{equation}
\label{eq:Zbasis}
\{Z_{j} \ : \ j \in \cJ \},\quad \cJ = \bigcup_{n\geq 0} \cJ_n
\end{equation}
where $\cJ$ is some set of indices and $\{Z_i \ : i \in \cJ_n \}$ is, for each  $n\geq 0$, a basis of the graded component of degree \( n\).
% *** under the assumption that there is an index \(e\in\cJ\) such that \(Z_e\) is the unit
% \(\one\) of $U(\tg)$.***SE USA?
Note that the structure constants \( \lambda^i_{i',i''}\) of the basis $\{G_i\ : \ i \in \cI\}$ of the Lie algebra $\tg$,
\[
[G_{i'},G_{i''}] = \sum_{i} \lambda^i_{i',i''} G_{i}, \quad i',i'' \in \cI,
\]
uniquely determine (see Section~\ref{ss:Delta}) the structure constants \(\mu^j_{j',j''}\) of the basis $\{Z_j\ : \ i \in \cJ \}$ of $U(\tg)$,
\begin{equation}\label{eq:mus}
Z_{j'} \star Z_{j''} = \sum_{j} \mu^j_{j',j''} Z_{j}, \quad j',j'' \in \cJ.
\end{equation}

\subsection{Constructing the Hopf algebra}
\label{ss:Hopf}

We now construct a Hopf algebra \(\H\) which will play in the present circumstances the role that the shuffle
Hopf algebra had in the preceding section.  The presentation that follows uses explicitly the choice of basis
in \eqref{eq:Zbasis}; this is convenient for the computational purposes we have in mind. However the  Hopf
algebra \(\H\) that we shall construct is in fact independent of the choice of basis,  {as shown in
Subsection~\ref{ss:Delta} below}. In the particular case where \(\tg\) is  freely generated by the elements of
a finite alphabet \(\A\), the construction below results in the shuffle Hopf algebra of the preceding section.

For each \(j\in\cJ\) we consider the linear  form \(u_j\) on \(U(\tg)\) that
takes the value \(1\) at the element \(Z_j\) and vanishes at each \(Z_{j'}\), \(j'\neq j \) and set
\(\H\) equal to the graded dual \(\bigoplus_{n\geq0} \H_n\) of $U(\tg)$, i.e.\ each \(\H_n\) is the subspace of the linear dual  $U(\tg)^*$ of $U(\tg)$ spanned by the
\(u_j\), \(j\in\cJ_n\).

We define a product in \(\H\) as follows. The algebra $U(\tg)$ possesses a canonical coalgebra structure whose coproduct
$\Delta:U(\tg) \to U(\tg) \otimes U(\tg)$ is uniquely determined by requiring that
\begin{itemize}
\item $\Delta(\beta) = \one \otimes \beta + \beta \otimes \one$, for all $\beta \in \tg$,
%\item $\Delta (\one) = \one \otimes \one$, and
\item $\Delta$ be an algebra homomorphism.
\end{itemize}
This coproduct is by construction cocommutative, i.e.\
if, for each $j \in \cJ$,
\[
\Delta(Z_{j} ) = \sum_{j',j'' \in \cJ} \eta^j_{j',j''} \, Z_{j'} \otimes Z_{j''}.
\]
then $\eta^j_{j',j''} =  \eta^j_{j'',j'}$.
 Through the duality between the vector spaces \(U(\tg)\) and \(\H\),
 \(\Delta\) induces the following commutative multiplication operation in $\H$:
 \[
 u_{j'} u_{j''} = \sum_{j \in \cJ}  \eta^j_{j',j''} \, u_{j} = \sum_{j \in \cJ} \langle \Delta(Z_{j}), u_{j'} \otimes u_{j''} \rangle \, u_{j} .
 \]

Similarly, the product \(\star\) in \(U(\tg)\) with structure constants given in \eqref{eq:mus}
induces by duality a coproduct $\Delta:\H \to \H \otimes \H$ given by
\begin{equation}
\label{eq:Delta}
\Delta (u_j) =  \sum_{ j',j'' \in \cJ} \mu^j_{j',j''} \  u_{j'} \otimes u_{j''},\qquad j \in \cJ
\end{equation}  (our hypotheses ensure that the summation
in \eqref{eq:Delta} has finitely-many non-zero terms and is therefore well defined). In this way \( \H\) is a connected, commutative, graded Hopf algebra.

We now turn to the dual vector space \( \H^*\). Each element $\gamma \in \H^*$ may be represented as a formal series
\[
\gamma = \sum_{j \in \cJ} \langle \gamma, u_{j} \rangle \, Z_j,
\]
where \(\langle \gamma, u_{j} \rangle\) is the image of \(u_j\in\H\) by the linear form \(\gamma\). Thus \(\H^*\) may be seen as a superspace of \(U(\tg)\).  The associative algebra structure  of $U(\tg)$ may be extended naturally to $\H^*$: for $\gamma',\gamma'' \in \H^*$, the series that represents their product \( \gamma = \gamma' \star \gamma'' \in \H^*\) is given by
\begin{eqnarray*}
\sum_{j \in \cJ} \langle \gamma, u_{j} \rangle \, Z_{j}  &=&
\left(\sum_{j' \in \cJ} \langle \gamma', u_{j'} \rangle \, Z_{j'} \right) \star
\left(\sum_{j'' \in \cJ} \langle \gamma'', u_{j''} \rangle \, Z_{j''} \right) \\
&=&
\sum_{j',j'' \in \cJ}  \langle \gamma', u_{j'} \rangle\,  \langle \gamma'', u_{j''} \rangle \, Z_{j'} \star Z_{j''}  \\
&=& \sum_{j',j'' \in \cJ}  \langle \gamma', u_{j'} \rangle\,  \langle \gamma'', u_{j''} \rangle \, \sum_{j\in \cJ}  \mu^j_{j',j''} Z_{j} \\
&=& \sum_{j \in \cJ} \left( \sum_{ j',j'' \in \cJ} \mu^j_{j',j''} \,  \langle \gamma', u_{j'} \rangle\,  \langle \gamma'', u_{j''} \rangle \right)\, Z_{j}.
\end{eqnarray*}
In other words
\[
 \langle \gamma, u_{j} \rangle=
 \sum_{ j',j'' \in \cJ} \mu^j_{j',j''} \,  \langle \gamma', u_{j'} \rangle\,  \langle \gamma'', u_{j''} \rangle
\]
i.e.\ the product \(\star\) in \(\H^*\) corresponds via duality to the coproduct \eqref{eq:Delta} in \(\H\).
The group of characters of \(\H\) and the Lie algebra of infinitesimal characters are
\[
\label{eq:G}
\G = \left\{\gamma \in \H^*\ : \  \langle \gamma, u_{j'} u_{j''} \rangle  = \langle \gamma, u_{j'} \rangle \,  \langle \gamma, u_{j''} \rangle \right\},
\]
and
\[
\g = \left\{\gamma \in \H^*\ : \  \langle \gamma, u_{j'} u_{j''} \rangle  = \langle \gamma, u_{j'} \rangle \,  \langle \one, u_{j''} \rangle +
\langle \one, u_{j'} \rangle \,  \langle \gamma, u_{j''} \rangle \right\},
\]
respectively.
These are related by a  bijection $\exp:\g \to \G$,
as we saw in the particular case considered in the preceding section.

The abstract initial value problem
\begin{equation}\label{eq:abstract2}
 \frac{d}{dt} \alpha(t) = \alpha(t)* \beta(t), \quad \alpha(0)=\one,
\end{equation}
with \(\beta(t)\) any given integrable curve in \(\g\)  possesses a solution that for each \(t\) is an element of \(\G\). This solution may be computed by finding its coefficients by recursion with respect to the grading. In particular this is so for the curve such that
$\Psi(\beta(t)) = F(t)$ for all $t \in \R$, whose existence we assumed at the beginning of this section. We next translate this result into a result for the operator problem.

\subsection{Back to the operator initial value problem}

The mapping \(\Psi\) may be defined on \(\H^*\supset U(\tg)\) as an algebra map from $\H^*$ to the direct product algebra $\prod_{n\geq 0} \mathrm{End}(\CC)$ sending each $\gamma \in \H^*$ to
\[
\Psi(\gamma) =
\sum_{n\geq 0} \sum_{j' \in \cJ_n} \langle \gamma, u_{j} \rangle \, \Psi(Z_{j})  =
\sum_{j \in \cJ} \langle \gamma, u_{j} \rangle \, \Psi(Z_{j}).
\]
 For the product of two
formal series of endomorphisms we have
\[
\left(\sum_{j' \in \cJ} \langle \gamma', u_{j'} \rangle \, \Psi(Z_{j'}) \right)
\left(\sum_{j'' \in \cJ} \langle \gamma'', u_{j''} \rangle \, \Psi(Z_{j''}) \right) =
\sum_{j \in \cJ} \langle \gamma, u_{j} \rangle \, \Psi(Z_{j}),
\]
where $\gamma = \gamma' \star \gamma'' \in \H^*$.

The solution \(\alpha(t)\)  of the abstract initial value problem
leads to the following formal solution of \eqref{eq:odeX} (a compact alternative to \eqref{eq:solpicard})
\[
X(t) = \sum_{j \in \cJ} \langle \alpha(t), u_j \rangle\, \Psi(Z_j),
\]
and we shall check presently that (\ref{eq:autprop}) is formally satisfied in order to obtain a formal solution of the initial value problem \eqref{eq:odex}.

From the definition of \(\Delta\), for arbitrary $\chi_1,\chi_2 \in \CC$,
\[
\Psi(Z_{j})(\chi_1 \chi_2) = \sum_{j',j'' \in \cJ} \eta^j_{j',j''} \, (\Psi(Z_{j'})\chi_1) \, (\Psi(Z_{j''})\chi_2)
\]
and it follows by duality that, for each \(\gamma \in \H^*\),
\[
\sum_{j \in \cJ} \langle \gamma, u_{j} \rangle \, \Psi(Z_{j})(\chi_1 \chi_2) =
\sum_{j',j'' \in \cJ} \langle \gamma, u_{j'} u_{j''} \rangle \, (\Psi(Z_{j'})\chi_1) \, (\Psi(Z_{j''})\chi_2).
\]
If, in particular, \( \gamma \in\G\), then the right-hand side of the last equality coincides with
\[
\left( \sum_{j' \in \cJ} \langle \gamma, u_{j'} \rangle \, \Psi(Z_{j'})\chi_1\right)
\left( \sum_{j'' \in \cJ} \langle \gamma, u_{j''} \rangle \, \Psi(Z_{j''})\chi_2\right),
\]
i.e.\ the formal series \(\sum_{j \in \cJ} \langle \gamma, u_{j} \rangle \, \Psi(Z_{j})\) is formally an automorphism. This is in particular true, for each \(t\), for the series \(X(t)\) above, since we know that \(\alpha(t)\in\G\).

\subsection{Averaging with more general Hopf algebras}
 An abstract problem of the form \eqref{eq:abstract2} with \(\beta(t)\) \(2\pi/\omega\)-periodic with values in \(\g\) may be averaged with the help of equation \eqref{eq:sanseb} exactly as we saw in the case of the shuffle Hopf algebra. The result may be then transferred, via the morphism \(\Psi\), to average periodically forced systems \eqref{eq:odex} .

\subsubsection{Averaging with decorated rooted trees}
\label{sss:rootedtrees}
 As an illustration we take up again the task of averaging \eqref{eq:ode5} but this time we work with the
 Grossman--Larson graded Lie algebra  of rooted trees~\cite{GrossmanLarson89} with vertices decorated by letters of the alphabet \(\A= \{a,b,c,d\}\). We  use once more \eqref{eq:Psi} and extend \(\Psi\) to a Lie algebra morphism from the Grossman-Larson Lie algebra to the Lie algebra of derivations \(\derCC\). In the construction above, \(\H\) is the Connes-Kreimer Hopf algebra of rooted trees and the group of characters \(\G\) is the Butcher group.

As an example we find, under the zero-mean condition and truncating the contributions of trees with four or more vertices:
\begin{align*}
\bar \beta^{[3]} &= \underline{a} + \frac{1}{\omega^2} \left(\textstyle
\frac{1}{4 }\, \underline{a[  b[b]]} - \frac{1}{ 4}\,  \underline{b[a b]}  {+\frac{1}{ 4}\,  b[b[a]]} + \frac{1}{4 }\, \underline{a[b^2]}  -\frac{1}{ 2}\,  \underline{b[a[b]]}\right. \\
& \textstyle
 - \frac{1}{8}\,  \underline{c[a[c]]}
{+ \frac{1}{16}\, a[c^2]  - \frac{1}{16}\,  c[a c] +\frac{1}{16 }\, a[c[c] ]+ \frac{1}{16}\, c[c[a]]  }\\
& \textstyle { + \frac{1}{36 }\, a[d[d]] + \frac{1}{36}\,  d[d[a]] - \frac{1}{36}\,  d[a d]
+ \frac{1}{36 }\, a[(d)^2] - \frac{1}{18}\,  d[  a[d]]  } \\
&  \textstyle { - \frac{1}{8}\,  b[b[c]] + \frac{1}{8}\,  b[b c] - \frac{1}{8}\,  c[b[b]]
+\frac{1}{ 4}\,  b[c[b]] - \frac{1}{8}\,  c[(b)^2] - \frac{1}{8}\,  c[b[b]]   } \\
& \textstyle { +\frac{1}{ 4}\,  b[c[b]] - \frac{1}{8}\,  c[(b)^2]
-\frac{1}{1 2}\,  b[c[d]] + \frac{1}{1 2}\,  c[b d] - \frac{1}{1 2}\,  d[c[b]]  } \\
& \textstyle  { + \frac{1}{8}\,  b[d[c]]  + \frac{1}{8}\,  c[d[b]]- \frac{1}{8}\,  d[b c]
  - \frac{1}{2 4}\,  c[b[d]]  + \frac{1}{2 4}\,  b[c d]}
   \textstyle \left.  {  -\frac{1}{2 4}\,  d[b[c]] }
\right).
\end{align*}
Here the notation for rooted trees is as follows:
\begin{itemize}
\item \(a\) denotes the one-vertex rooted tree where the root is decorated with the symbol \(a\),
\item
\(d[b[c]]\) denotes the \lq tall\rq\ rooted tree where  the decoration \(d\) corresponds to the root, the vertex decorated by \(b\) is linked to the root, and the vertex decorated with \(c\) is linked to the vertex decorated with \(b\),
\item  \(d[b c]\) denotes the \lq bushy\rq\ rooted tree where \(d\) is the decoration of the root and the vertices with decoration \(b\) and \(c\) are linked to the root, etc.
\end{itemize}

As in the case of words, results on the abstract problem are transferred to Euclidean space with the help of \(\Psi\). We again find a system
\( (d/dt) \bar x = \bar f(\bar x)\), where \(\bar f\) is a formal series of vector fields in \(\R^5\) and a \(2\pi/\omega\)-periodic formal change of variables
\(x = U(\bar x,t)\) such that solutions \(x(t)\) of \eqref{eq:ode5} are formally given as \( x(t) = U(\bar x(t),t)\). Now the formal series are indexed by rooted trees rather than by words, i.e.\ they are B-series \cite{part1}, \cite{part2}, \cite{china}.

What is the advantage of using rooted trees rather than words? The expression for \(\bar \beta^{[3]}\) displayed above, with 33 rooted trees, is obviously more involved than its counterpart with words involving 19 words. However the images by \(\Psi\) of many trees vanish. For instance, in the display above only the five rooted trees underlined have a nonzero image. This may be exploited by working in
the quotient by $\mathrm{ker} (\Psi)$ of the Lie algebra of rooted trees,
thereby decreasing the dimension of the graded components, which allows symbolic manipulation packages to take the expansions to higher order. A further reduction may be achieved by noting that \(\bar\beta\) has to be a Lie element, i.e.\ it must be expressible in terms of commutators. We may then work in the Lie subalgebra  generated by \(a\),\dots, \(d\) of the previous quotient subalgebra. For instance for the display above we find the compact expression
\[
a + \frac{1}{\omega^2}  \left(\textstyle \frac{1}{4 }\, [b,[b,a]]  - \frac{1}{8}\, [c,[a,c]]\right).
\]

\subsubsection{Averaging in a Lie algebra generated by monomial vector fields}
\label{sec:monomials}

We have just seen how to work in a Lie algebra better suited to the concrete example at hand than the Lie algebra corresponding to words. Another possibility in this direction is to use a graded Lie algebra generated by monomial vector fields.  In our example, we consider the graded Lie algebra $\tg=\bigoplus_{n\geq 1}\g_n$ of vector fields generated by the monomial vector fields
\begin{eqnarray*}
&&U \, \partial_Y, U \, \partial_Z, V \, \partial_V, V \, \partial_Y,
X^3 \, \partial_Y,  X^2 \, \partial_Y,
 X \, \partial_Y, Y \, \partial_V,Y \, \partial_X,Y
   \, \partial_Y,Z \, \partial_U,\, \partial_V,\, \partial_X,\, \partial_Y,
\end{eqnarray*}
each of them belonging to $\g_1$. That is, we may multiply each of the monomial vector fields above by a bookkeeping parameter $\epsilon$, so that monomial vector fields affected by a $n$-th power of $\epsilon$ belongs to $\g_n$. In that case, the map $\Psi:\tg \to \derCC$ corresponds to replacing $\epsilon$ by 1.
With the help of this graded Lie algebra a symbolic package in a laptop computer may carry the computations necessary to perform \(n\)-order  averaging up to \( n = 16\), while, as mentioned above, with words we could not go beyond \(n = 8\).

\subsubsection{{Summary}}
{ The technique in \cite{part2} or \cite{kurusch} summarized in Section~\ref{sec:an example} averages
oscillatory differential systems like \eqref{eq:ode5} by first reformulating them in an abstract form
\eqref{eq:absivp} that is integrated in the group of characters \(\G\)  of the shuffle Hopf algebra.  The
solution of the abstract problem is then averaged and the result transferred back to the original system.
While the technique is completely general, its computational complexity grows very quickly with the required
accuracy. We have just seen that, by working with alternative Hopf algebras,  it is possible to diminish the
computational cost and achieve substantially higher orders of accuracy in  a given computing environment. }

\subsection{{Explicit construction of the coproduct $\Delta$}}
\label{ss:Delta}

In this subsection, we focus on determining, in a form suitable for actual computations, the coproduct $\Delta$ of the Hopf algebra constructed in Subsection~\ref{ss:Hopf} from a given graded Lie algebra $\tg$.

In the particular case where $\tg$ is the free Lie algebra generated by an alphabet $\A$,  $U(\tg)$ is isomorphic to the algebra \(\RA\). It therefore possesses a basis  (\ref{eq:Zbasis}) with $\mathcal{J}$ given by the set of words on the alphabet $\A$ (the operation $\star$ corresponds to the concatenation of words). The coproduct $\Delta$ of $\H$ expressed in that basis indexed by words is then the deconcatenation coproduct, which has
 a particularly simple form.

For an arbitrary graded Lie algebra $\tg$, the coproduct $\Delta:\H \to \H \times \H$ can be uniquely determined from the structure constants $\lambda^i_{i',i''}$ of a basis (\ref{eq:Gi}) of $\tg$.
Recall that  the Poincar\'e-Birkhoff-Witt  (PBW) basis of $U(\tg)$ is a basis (\ref{eq:Zbasis}) indexed by the set
\begin{equation}
\label{eq:J}
\cJ =\{e\} \cup  \{(i_1,\ldots,i_m) \in \cI^m \ : \  m\geq 1, \   i_1 \leq \cdots \leq i_m \},
\end{equation}
where, as above, \(\cI\) is the set of indices for the homogeneous basis of the graded Lie algebra \(\tg\) and
\( \cI^m \) is the product \( \cI\times \cdots \times \cI \) (\( m \)-times).
 The empty index $e$ is associated with the unit $\one$ of $U(\tg)$, that is $Z_{e}=\one$.
 For $j=(i_1,\ldots,i_m) \in \cJ$, the elements $Z_j$ are defined by (\ref{eq:PWBbasis}) scaled by the inverse of the product of some factorials. More precisely,
 $Z_j = 1/j! \, G_{i_1} \star \cdots \star G_{i_m}$, where $j!=m!$ if $i_1=i_2=\cdots = i_m$, and $j!=k! (i_{k+1},\ldots,i_m)!$ if $i_1=\cdots = i_k < i_{k+1}$.
It is well known~\cite{Bourbaki} that the basis $\{u_j \ : \ j \in \cJ\}$ of $\H$ dual to the PWB basis of  $U(\tg)$ satisfies that $u_j=v_{i_1} \cdots v_{i_m}$ for $j=(i_1,\ldots,i_m) \in \cJ$,  where $v_i := u_{(i)}$ for each $i \in \cI$. We thus have that, as an algebra, $\H$ is a polynomial algebra on the commuting indeterminates $\{v_i\ : \ i \in \cI\}$, that is, the symmetric algebra $S(V)$ over the vector space  $V$ spanned by $\{v_i\ : \ i \in \cI\}$.

Since $\Delta:\H \to \H\otimes\H$ is an algebra map, it is enough to determine $\Delta(v_i) \in \H \otimes \H$ for $i \in \cI$ from the structure constants $\lambda^i_{i',i''}$. However, existing algorithms for that task  are rather involved. Fortunately, there are bases of $U(\tg)$ that are computationally more convenient than the PWB basis for our purposes.
This may be illustrated  for the Grossman-Larson graded Lie algebra $\tg$ considered in Subsection~4.4: it has
a basis (\ref{eq:Gi}) indexed by the set $\cI:=\mathcal{T}$ of rooted trees decorated by the letters of the alphabet $\A=\{a,b,c,d\}$  providing a simple description of the Lie bracket in terms of grafting of rooted trees~\cite{GrossmanLarson89}.
 In that case, the index set (\ref{eq:J}) can be identified with the set $\mathcal{F}$ of forest of rooted trees over $\A$.
 As noted before, the corresponding commutative Hopf algebra $\H$ is the Connes-Kreimer Hopf algebra over the alphabet $\mathcal{A}$.  The construction of $\H$  described above in terms of the PWB basis of $U(\tg)$ realizes the Hopf algebra $\H$ as the polynomial algebra on the commuting indeterminates $\{v_i\ : \ i \in \mathcal{T}\} \subset U(\tg)^*$.  However, the expresions of $\Delta(v_i)$ for $i \in \mathcal{T}$ in that representation of $\H$ is rather cumbersome, and fails to reflect the nice combinatorial nature of the coproduct of the Connes-Kreimer Hopf algebra.

The task of determining the commutative graded Hopf algebra $\H$ from  the
structure constants $\lambda^i_{i',i''}$ of a basis (\ref{eq:Gi}) of $\tg$ can be reformulated in terms of a graded Lie coalgebra~\cite{Michelis1980} structure $(V,\delta)$ related to the graded Lie algebra $\tg$ as follows.
Let $V=\bigoplus_{n\geq 1} V_n$ be a
graded vector space  with a homogeneous basis
$\{v_i\ : \ i \in \cI\}$, and consider the graded linear map $\delta: V \to V \otimes V$ defined by
\begin{equation}
\label{eq:delta}
\delta(v_i) =  \sum_{i',i'' \in \cI} \lambda^i_{i',i''} \, v_{i'} \otimes v_{i''},  \mbox{ for }  i \in \cI.
\end{equation}
Since the coefficients $\lambda^i_{i',i''}$ are the structure constants with respect to a basis (\ref{eq:Gi}) of the graded Lie algebra $\tg$, $(V,\delta)$ is by construction a graded Lie coalgebra. The dual map $\delta^*:V^* \otimes V^* \to V^*$ endows the linear dual $V^*$ with a structure of Lie algebra such that
% isomorphic to the Lie algebra $\g$ of infinitesimal characters of the Hopf algebra $\H$. Hence,
$\tg$ is isomorphic to a Lie subalgebra of the Lie algebra $V^*$.

Now, our original task can be formulated as follows: find an algebra map $\Delta:S(V) \to S(V) \otimes S(V)$ satisfying the following two conditions:
\begin{itemize}
\item the coproduct $\Delta$ endows the symmetric algebra $S(V)$ with a graded connected Hopf algebra structure $\H$,
\item the linear map $\hat \delta: V \to V \otimes V$ such that, for each $v \in V$, $\hat \delta(v)$ is the projection of $\Delta(v)$   onto $V \otimes V$  satisfies the relation
\begin{equation}
\label{eq:deltahdelta}
\delta = \hat \delta - \tau \circ \hat \delta,
\end{equation}
where $\tau:V\otimes V \to V \otimes V$ is defined by $\tau(v\otimes v') = v' \otimes v$.
\end{itemize}
Such an algebra map $\Delta$ is not unique,  but the corresponding coalgebra structure on $S(V)$ is unique up to isomorphisms.\footnote{It is actually the universal coenvelopping coalgebra~\cite{Michelis1980} of the Lie coalgebra $V$.}  %The non-uniqueness of $\Delta$ can be seen as follows. If $\Delta$ satisfies the required conditions, then consider $\Delta'$ defined so that $\Delta'(v_i) = \Delta(v_i) + \Delta(w_i)$, where for each $i \in \cI$, $w_i$ is an element of homogeneous degree $|i|$ such that the projection of $\Delta(w_i)$ to $V \otimes V$ vanish.

Observe that here there is no need to assume that the homogeneous vector subspaces $\tg_n$ are finite dimensional. One only needs to assume that the Lie coproduct (i.e.\ Lie cobracket) $\delta$ is well defined, or in other words, that the structure constants $\lambda^i_{i',i''}$ of the Lie bracket with respect to a basis (\ref{eq:Gi}) are such that, for each $i \in \cI$, the sum in (\ref{eq:delta}) is well defined.

Notice also that the choice of the basis for $V$ plays no role in this formulation in terms of the Lie coalgebra $(V,\delta)$.  It can be shown that, for a given basis (\ref{eq:Gi}) of the Lie algebra $\tg$,
 each choice of $\Delta$ gives rise, after dualization, to a different basis (\ref{eq:Zbasis}) (indexed by the set (\ref{eq:J})) of $U(\tg)$

We will next make use of the concept of pre-Lie algebra. (We refer to~\cite{Manchon2011} for a survey on pre-Lie algebras.)
Assume  now that there exists a graded linear map $\hat \delta: V \to V \otimes V$ satisfying (\ref{eq:deltahdelta}). According to Proposition 3.5.2. in~\cite{GS2008} (see also Theorem~5.8 in~\cite{LR2010}),  if  $(V,\hat \delta)$ is a graded pre-Lie coalgebra (that is, $(V^*,\hat \delta^*)$ is a pre-Lie algebra) then there exists a graded  algebra map $\Delta:S(V) \to S(V) \otimes S(V)$ satisfying the following conditions:
\begin{itemize}
\item the coproduct $\Delta$ endows the symmetric algebra $S(V)$ with a graded connected Hopf algebra structure $\H$,
\item for each $v \in V$, $\Delta(v)- 1 \otimes v - v \otimes 1 \in S(V) \otimes V$,
\item for each $v \in V$,  $\hat \delta(v)$ is the projection to $V \otimes V$ of $\Delta(v)$.
\end{itemize}
(Actually, the converse also holds~\cite{GS2008}.) Furthermore such an algebra map $\Delta$ is uniquely determined by $\hat \delta$. In ~\cite{GS2008}, a recursive procedure to determine $\Delta(v)$ for each $v \in S(V)$ in terms of the pre-Lie coproduct $\hat \delta$ is presented. In Theorem~\ref{th:hdeltaDelta} below, we suggest an alternative recursive procedure.

It is worth mentioning that the dual basis of the basis of monomials $v_{i_1} \cdots v_{i_m}$ of $S(V)$ corresponding to the uniquely determined coproduct $\Delta$ is precisely the basis of the universal enveloping algebra $U(\tg)$ of the pre-Lie algebra $\tg$ considered in~\cite{OG2005}.

Coming back to the Grossman-Larson graded Lie algebra $\tg$ of rooted trees over an alphabet $\A$, it is known that it is the free pre-Lie algebra over the set $\A$~\cite{chapoton}. The Lie algebra morphism $\Psi:\tg \to \derCC$ considered in Subsection~\ref{sss:rootedtrees} is actually the unique extension of \eqref{eq:Psi} to a  pre-Lie algebra morphism from the Grossman-Larson Lie algebra over the alphabet $\{a,b,c,d\}$ to the Lie algebra of derivations \(\derCC\). The corresponding pre-Lie coproduct $\hat \delta: V \to V \otimes V$
can be nicely described in terms of all the splittings of the rooted tree in two parts by successively removing each of the edges. The coproduct $\Delta: V \to S(V) \otimes S(V)$ uniquely determined in Proposition 3.5.2 of~\cite{GS2008} coincides  with the Connes-Kreimer coproduct defined in terms of the so called admissible cuts of rooted trees and forests.

This construction of the Hopf algebra $\H$ from a pre-Lie coalgebra structure $(V,\hat \delta)$ may seem rather restrictive. However,
any graded Lie coalgebra $V=\bigoplus_{n\geq 1} V_n$ with Lie coproduct  $\delta$ admits at least one pre-Lie coproduct $\hat \delta$ satisfying (\ref{eq:deltahdelta}), as we will show later on.

Let $\H=S(V)$ be the graded connected commutative Hopf algebra uniquely determined by a given pre-Lie coalgebra $(V,\hat \delta)$, with coproduct $\Delta:\H \to \H \otimes \H$ and antipode $S:\H \to \H$.
We define the grading operator $\rho: \H \to \H$ given by $\rho(u) = n \, v$ if $u$ belongs to the graded component $\H_n$.
 Furthermore, we define the derivation $\partial:\H^* \to \H^*$ as  the dual map of $\rho$, i.e.\
$\langle \partial(\gamma), u \rangle = \langle \gamma, \rho(u) \rangle$ for each $\gamma \in \H^*$ and each $u \in \H$.

We also use the (generalized) Dynkin operator $D$ as considered in~\cite{EFGBP2007}, \cite{MKL2011} for graded connected commutative Hopf algebras. (See~\cite{PR2003}, \cite{MP2013} and references therein for the generalized Dynkin operator in the cocommutative case.) The Dynkin operator $D:S(V) \to V$ of  $\H=S(V)$ is the convolution $D:=\rho * S$ (in the references above, $D$ is actually defined as $S * \rho$) of the antipode and the grading operator,  that is,
\begin{equation*}
D:=\mu_{\H} \circ (\rho \otimes S) \circ \Delta,
\end{equation*}
where $\mu_{\H}:\H \otimes \H \to \H$ is the multiplication map of the algebra $\H=S(V)$. It is not difficult to check that the convolution of $D$ with the identity $\mathrm{id}_{\H}$ in $\H$ coincides with $\rho$, that is,
\begin{equation}
\label{eq:rhoD}
\rho=\mu_{\H} \circ (D \otimes \mathrm{id}_{\H}) \circ \Delta.
\end{equation}

The Dynkin operator has the property that
\begin{equation}
\label{eq:D(u)}
D(u)=0 \quad \mbox{for all} \quad u \in V^2S(V),
\end{equation}
and we thus have that, for each $v \in V$,
\begin{equation}
\label{eq:l1}
(D \otimes \mathrm{id}_{\H}) \circ \Delta(v) = D(v) \otimes 1 + (D \otimes \mathrm{id}_{V}) \circ \hat \delta(v).
\end{equation}
This in turn implies, together with (\ref{eq:rhoD}) the following result, which allows to inductively determine $D(v)$ for each $v \in V$ in terms of the pre-Lie coproduct $\hat \delta$.
\begin{lemma}
\label{l:1}
%Let $(V,\hat \delta)$ be a graded pre-Lie coalgebra.
For each $v \in V$,
\begin{equation}
\label{eq:D(v)}
D(v) = \rho(v) - \mu_{\H} \circ (D \otimes \mathrm{id}_{V}) \circ \hat \delta(v).
\end{equation}
\end{lemma}
\begin{theorem}
 \label{th:hdeltaDelta}
 Let $(V,\hat \delta)$ be a graded pre-Lie coalgebra, and consider the linear map $D:S(V)\to V$ determined by (\ref{eq:D(u)}) and (\ref{eq:D(v)}).  The symmetric algebra $\H=S(V)$ becomes a
 graded connected Hopf algebra with the coproduct $\Delta$ determined as the unique graded algebra map $\Delta:\H \to \H \otimes \H$ such that
 \begin{equation}
\label{eq:Delta(v)}
\Delta(v) = 1 \otimes v +  v \otimes 1 +\bar \Delta(v), \quad v \in V,
\end{equation}
where the linear map $\bar \Delta: V \to \H \otimes V$ is  uniquely determined by the identity
\[
(\rho \otimes \id_V) \circ \bar \Delta = (D-\rho) \otimes 1 +  ( \mu_{\H} \otimes \id_V )  \circ ( D \otimes \Delta) \circ \hat \delta.
\]
\end{theorem}
\begin{proof}
From (\ref{eq:Delta(v)}) one has that
\begin{equation*}
(\rho \otimes \id_V) \circ \bar \Delta = -(\rho(v) \otimes 1) + (\rho \otimes \id_V) \circ  \Delta
\end{equation*}
Application of  (\ref{eq:rhoD}), the coassociativity of $\Delta$, and (\ref{eq:l1}) lead to
\begin{align*}
(\rho \otimes \id_V) \circ  \Delta &=
(\mu_{\H} \otimes \id_V) \circ ((D \otimes \id_{\H}) \circ \Delta ) \otimes \id_V) \circ \Delta(v) \\
&=
(\mu_{\H} \otimes \id_V) \circ (D \otimes \id_{\H} \otimes \id_{\H}) \circ (  \Delta  \otimes \id_V) \circ \Delta(v)\\
&=
(\mu_{\H} \otimes \id_V) \circ (D \otimes \id_{\H} \otimes \id_{\H}) \circ (  \id_V  \otimes \Delta) \circ \Delta(v) \\
&=
(\mu_{\H} \otimes \id_V) \circ (  \id_V  \otimes \Delta) \circ
 (D \otimes \id_{\H})  \circ \Delta(v)\\
&=
(\mu_{\H} \otimes \id_V) \circ (  \id_V  \otimes \Delta) \circ
 (D(v) \otimes 1 + (D \otimes \mathrm{id}_{V}) \circ \hat \delta(v))\\
&=
(\mu_{\H} \otimes \id_V) \circ (  D(v)  \otimes 1 \otimes 1)  +
(\mu_{\H} \otimes \id_V) \circ
  (D \otimes \Delta) \circ \hat \delta(v)\\
  &=
 (  D(v)   \otimes 1)  +
(\mu_{\H} \otimes \id_V) \circ
  (D \otimes \Delta) \circ \hat \delta(v).
 \end{align*}
 \qed
\end{proof}

Given a graded Lie coalgebra   $(V,\delta)$, consider the graded linear map  $\hat \delta:V \to V \otimes V$ determined in terms of $\delta$ by
\begin{equation}
\label{eq:hdelta}
\rho \circ \hat \delta = (\id_V \otimes \rho) \circ \delta,
\end{equation}
 that is,
\[
\hat \delta(v_i) =
\sum_{i',i'' \in \cI} \frac{|i''|}{|i|} \lambda^i_{i',i''} \, v_{i'} \otimes v_{i''},  \mbox{ for }  i \in \cI.
\]
Clearly, (\ref{eq:deltahdelta}) holds, and it is not difficult to check that $(V,\hat \delta)$ is a pre-Lie coalgebra, or equivalently, that the binary operation $\rhd:V^* \otimes V^* \to V^*$ obtained by dualizing the coproduct $\hat \delta$ endows $V^*$ with a structure of graded pre-Lie algebra. Indeed, for each $\beta',\beta'' \in V^*$,
\[
\beta' \rhd \beta'' =\partial^{-1} [\partial(\beta'),\beta''], \quad
\beta',\beta'' \in \tg,
\]
where $\partial^{-1}$ denotes the inverse of the restriction to $V^*$ of $\partial$, so that $[\beta',  \beta''] = \beta' \rhd \beta'' - \beta'' \rhd \beta'$.

\begin{theorem}
 \label{th:deltaDelta}
 Let $(V,\delta)$ be a graded pre-Lie coalgebra.   The symmetric algebra $\H=S(V)$ becomes a
 graded connected Hopf algebra with the coproduct $\Delta$ determined as the unique graded algebra map $\Delta:\H \to \H \otimes \H$ such that (\ref{eq:Delta(v)}) holds and  the linear map $\bar \Delta: V \to \H \otimes V$ is  uniquely determined by the relation
\[
(\rho \otimes \id_V) \circ \bar \Delta = ( \mu \otimes \id_V )  \circ ( \rho \otimes \Delta) \circ \hat \delta,
\]
where $\hat \delta:V \to V \otimes V$ is determined by (\ref{eq:hdelta}).
\end{theorem}
\begin{proof}
The definition  (\ref{eq:hdelta}) of $\hat \delta$ and the assumption of $(V,\delta)$ being a Lie coalgebra implies that the identity (\ref{eq:D(v)}) holds with $D(v):=\rho(v)$. The result then follows from Theorem~\ref{th:hdeltaDelta}.
\qed
\end{proof}

Finally, we provide the dual basis  (\ref{eq:Zbasis}) (indexed by the set (\ref{eq:J})) of the basis of monomials $v_{i_1} \cdots v_{i_m}$ of $S(V)$ corresponding to the graded connected commutive Hopf algebra structure on $S(V)$ determined in Theorem~\ref{th:deltaDelta}.

 We first need some notation.
  \begin{itemize}
\item Given $i \in \cI$, we  write $|i|=n$ if $i \in \cI_n$. For  $j=(i_1,\ldots,i_m) \in \cJ$, we set $|j| = |i_1| + \cdots + |i_m|$. We also write $|e|=0$.
 \item  Given $j =(i_1,\ldots,i_m)\in \cJ$ and $i \in \cI$, we write $i \in j$ if $i \in \{i_1,\ldots,i_m\}$, and, in that case, we denote as $(j\backslash i)$ the element of $\cJ$ obtained by removing from $j=(i_1,\ldots,i_m)$ one occurrence of $i$. In particular, if $j=(i)$, then $(j\backslash i)=e$.
\end{itemize}
For $j=(i_1,\ldots,i_m) \in \cJ$, we set:
 \begin{equation}
\label{eq:Zj}
Z_{j} = \frac{|i|}{|j|} \sum_{i \in j}  G_{i} \star Z_{(j\backslash i)}.
\end{equation}
Observe that $Z_{(i)} = G_{i}$ for all $i \in \cI$.
\begin{theorem}
The set (\ref{eq:Zbasis}) of elements of $U(\tg)$ given by (\ref{eq:Zj}) is a basis of $U(\tg)$ dual to the basis of monomials $u_j = v_{i_1} \cdots v_{i_m}$ for $j=(i_1,\ldots,i_m) \in \cJ$ of the Hopf algebra determined in Theorem~\ref{th:deltaDelta}.
\end{theorem}
\begin{proof}
For each character $\alpha \in \G$ of $\H$, it holds~\cite{MKL2011}
\[
\langle \partial(\alpha)\star \alpha^{-1}, u \rangle =
\langle \alpha, D(u) \rangle \quad \mbox{for all} \quad u \in \H.
\]
Since $D(v)=\rho(v)$ for all $v \in V$,
\[
\langle \partial(\alpha)\star \alpha^{-1}, v \rangle =
\langle \alpha, \rho(v) \rangle.
\]
The later is equivalent to
\[
\partial\left(\sum_{j \in \cJ} \langle \alpha, u_j \rangle \, Z_{j}\right) =
\partial\left(\sum_{i \in \cI} \langle \alpha, v_i \rangle \, G_{i} \right) \star \left(\sum_{j \in \cJ} \langle \alpha, u_j \rangle \, Z_{j}\right).
\]
One finally arrives to (\ref{eq:Zj})
by expanding the right-hand side of that identity and equating terms.
\qed
\end{proof}

\section{Perturbed problems}

As we saw in Section~\ref{sec:an example}, the use of the shuffle Hopf algebra to average oscillatory problems in Euclidean space leads to word series expansions. {\em Extended word series}, introduced in \cite{words}, are a generalization of word series which appear in a natural way when solving some problems by means of the techniques we are studying. These include the reduction to normal form of continuous or discrete dynamical systems \cite{words}, \cite{juanluis}, the analysis of splitting algorithms of perturbed integrable problems \cite{words}, the computation of formal invariants of perturbed Hamiltonian problems \cite{juanluis} and averaging of perturbed problems \cite{kurusch}. We now study the extension of these techniques to scenarios where the shuffle Hopf algebra is replaced  by other Hopf algebras.

 We consider the situation where in the initial value problem  (\ref{eq:odeX}),  $F(t)$ is a perturbation $F(t) = F^0 + \tilde F(t)$ of a derivation $F^0 \in \derCC$ with a well defined exponential curve $\exp(t \, F^0)$ in $\autCC$.
If the solution $X(t)$ of the given problem
\begin{equation}
\label{eq:odeXp}
\frac{d}{dt} X(t) = X(t)  ( F_0 + \tilde F(t)), \quad X(0)=I,
\end{equation}
exists, then it may be written as $X(t) = Y(t) \exp(t \, F_0)$, where the curve $Y:\R \to \autCC$ is the solution of the initial value problem
\begin{equation}
\label{eq:odeY}
\frac{d}{dt} Y(t) = Y(t) \exp(t \, F_0) \tilde F(t) \exp(-t \, F_0), \quad Y(0)=I.
\end{equation}
\subsection{Algebraic framework for perturbed problems}
\label{sec:productgroup}

We  work  with a graded Lie algebra
\begin{equation}
\label{eq:gradedLA0}
\bigoplus_{n\geq 0} \g_n
\end{equation}
with finite-dimensional homogeneous subspaces $\g_n$, and a Lie group $\G_0$ with Lie algebra $\g_0$ such that the exponential map $\exp:\g_0 \to \G_0$ is bijective;  we observe that $\g_0$  and
\[
%\label{eq:tg}
 \tg = \bigoplus_{n\geq 1} \g_n,
\]
are respectively a Lie subalgebra and a Lie ideal of (\ref{eq:gradedLA0}) and denote by \( \G_0\) a Lie group with Lie algebra \(\g_0\).

Under such assumptions, one can prove that there exists an action $\cdot$ of the group $\G_0$  on the Lie algebra $\tg$ that is homogeneous of degree 0 (i.e.\ its restriction to each $\g_n$ is an action on $\g_n$), and  such that, for arbitrary $\tilde \beta \in \bigoplus_{n\geq 1} \g_n$ and $\beta_0 \in \g_0$, $\alpha_0(t)=\exp(t\, \beta_0)$,
\begin{equation}
\label{eq:actionDE}
\frac{d}{dt} \left( \alpha_0(t)\cdot \tilde \beta \right) = [\beta_0, \alpha_0(t)\cdot \tilde \beta].
\end{equation}

We  consider the commutative graded connected Hopf algebra $\H=\bigoplus_{n\geq 0} \H_n$ associated with the graded Lie algebra $\tg$, its group of characters $\G \subset \H^*$, and its Lie algebra of infinitesimal characters $\g \subset \H^*$.

 For  $\beta \in \g_0$, $\ad_{\beta}=[\beta,\cdot]$ is a derivation of (homogeneous degree 0 of) the graded Lie algebra (\ref{eq:gradedLA0}). Its restriction to $\tg$ is also a derivation of the Lie subalgebra $\tg$. This derivation can be extended to a derivation of the Lie algebra $\g$ of infinitesimal characters of $\H$. Hence, one can construct the semidirect sum Lie algebra
\[
\bar \g := \g \oplus_{S} \g_0 \supset \bigoplus_{n\geq 0} \g_n.
\]
More specifically, given $\bar \beta = \beta_0 +   \beta \in \bar \g$ and  $\bar\beta' = \beta'_0 +   \beta' \in \bar \g$ (where $\beta_0,\beta_0' \in \g_0$ and $  \beta,   \beta' \in \tg$), then
\[
[\bar \beta, \bar \beta'] = [\beta_0, \beta_0'] + (\ad_{\beta_0}   \beta'  - \ad_{\beta'_0}   \beta + [  \beta,   \beta']),
\]
\

The action of $\G_0$ on $\tg$ can be extended to an action of $\G_0$ on $U(\tg)$ and from that to an action on $\H^*$.  In particular, this defines an action of $\G_0$ on $\G$, which allows us to consider the semidirect product group
\[
\bar \G:=\G \ltimes \G_0.
\]
 More specifically, let  $(\alpha, \alpha_0), (\alpha', \alpha'_0) \in \bar \G$ (where $\alpha_0,\alpha'_0 \in \G_0$ and $  \alpha,   \alpha' \in \G$), then the product law $\circ$ in $\bar G$ is defined in terms of the action $\cdot$ and the product laws $\circ$ and $\star$ of $\G_0$ and $\G$ respectively as
\[
( \alpha, \alpha_0) \circ (\alpha',\alpha'_0) = (\alpha \star (\alpha_0 \cdot \alpha), \alpha_0 \circ \alpha'_0).
\]
We identify $(\one,\G_0)$ with $\G_0$, and $(\G,\id_0)$ with $\G$ (here $\id_0$ denotes the neutral element in the Lie group $\G_0$); then we write the elements $(\alpha,\alpha_0) \in \bar \G$ as $\alpha \circ \alpha_0$. In particular,  $\alpha_0 \cdot \alpha = \alpha_0 \circ \alpha \circ \alpha_0^{-1}$.
We denote as $\id$ the identity element in $\bar \G$.

Given a smooth curve $\bar \alpha:\R \to \bar \G$ such that $\bar \alpha(0)=\id$, its derivative at $t=0$ is
\[
\left. \frac{d}{dt} \bar \alpha(t)\right|_{t=0} :=
\left. \frac{d}{dt}  \alpha(t)\right|_{t=0} + \left. \frac{d}{dt}  \alpha_0(t)\right|_{t=0},
\]
with $\bar \alpha(t) = \alpha(t) \circ \alpha_0(t)$, where for all $t\in \R$,
$\alpha(t) \in \G$, $\alpha_0(t) \in \G_0$. This can be used to define
the adjoint representation $\mathrm{Ad}:\bar \G \to \mathrm{Aut}(\bar \g)$. In particular, for $\alpha_0 \in \G_0$, $\beta \in \g$,
$\mathrm{Ad}_{\alpha_0} \beta = \alpha_0 \cdot \beta$.

The exponential map $\exp:\bar \g \to \bar \G$ is defined as follows: given $\bar\beta = \beta_0 + \beta \in \bar \g$,  then $\exp(\beta_0 + \beta) := \alpha(1) \circ \exp(\beta_0)$, where
$\alpha(t) \in \G$ is the solution of (\ref{eq:abstract2}) with $\beta(t)$ replaced by $\exp(t \beta_0) \cdot \beta$. With this definition,
$\{\exp(t\, (\beta_0 + \beta)) \ : \ t \in \R\}$ is a one-parameter subgroup of $\bar \G$, and
\(
\left. (d/dt) \exp(t\, (\beta_0 + \beta)) \right|_{t=0} =\beta_0 + \beta.
\)

In general $\exp:\bar \g \to \bar \G$ is not surjective \cite{juanluis}.  Given $\bar \alpha = \alpha \circ \exp(\beta_0) \in \bar G$, there exists $\beta \in \g$ such that $\bar \alpha = \exp(\beta_0 + \beta)$ if, for each $n \geq 1$, the restriction to $\g_n$ of $\int_0^1 \Ad_{ \exp(t \beta_0)} dt$ is invertible. (The importance of this hypothesis will be illustrated in Subsection~\ref{ss:splitting} below.)

\subsection{Back to perturbed differential equations}

We now consider a perturbed operator differential equation (\ref{eq:odeXp}),
and assume that there exist a Lie algebra homomorphism
\[
\Psi:\bigoplus_{n\geq 0} \g_n \to \derCC,
\]
an element  $\beta_0 \in \g_0$ and a curve $\tilde \beta(t)$ in $ \tg$
with $\Psi(\beta_0)=F_0$ and $\Psi(\tilde \beta(t)) = \tilde F(t)$.
This together with (\ref{eq:actionDE}) implies that
\[
\Psi(\exp(t \, \beta_0) \cdot  \tilde \beta(t)) = \exp(F^0) \tilde F(t) \exp(-F^0).
\]
Equation (\ref{eq:odeY}) now reads
\begin{equation}
\label{eq:odeY2}
\frac{d}{dt} Y(t) = Y(t)  \Psi(\alpha_0(t) \cdot \tilde \beta(t)), \quad Y(0)=I,
\end{equation}
where $\alpha_0(t) = \exp(t \beta_0)$.
 The problem (\ref{eq:odeY2}) can be formally solved with the techniques in the preceding section as
 \[
Y(t)= \sum_{j \in \cJ} \langle \alpha(t), u_j \rangle\, \Psi(Z_j),
\]
where $\alpha:\R \to \G$ is the solution of
\[
\frac{d}{dt} \alpha(t) = \alpha(t) \star (\alpha_0(t) \cdot \tilde \beta(t)), \quad \alpha(0)=\one.
\]
 Hence, a formal solution $X(t)$ of (\ref{eq:odeXp}) is given by
\[
X(t) = \left( \sum_{j \in \cJ} \langle \alpha(t), u_j \rangle\, \Psi(Z_j)  \right)\exp(t \, \Psi(\beta_0)).
\]

As expected, the map that sends each $\bar \alpha =\alpha \circ \exp(\beta_0) \in \bar \G$ to the formal automorphism
\[
\left( \sum_{j \in \cJ} \langle \alpha, u_j \rangle\, \Psi(Z_j)  \right)\exp(t \, \Psi(\beta_0))
\]
behaves as a group homomorphism. Similarly, the map that sends each $\bar \beta = \beta_0 + \beta \in \bar \g$ to the formal derivation
\(
\Psi(\beta_0) +
 \sum_{j \in \cJ} \langle \beta, u_j \rangle\, \Psi(Z_j)
\)
behaves as a Lie algebra homomorphism.
In addition, if $\exp(\beta_0 + \beta) = \alpha \circ \exp(\beta_0)$, then
\[
\exp\left(
\Psi(\beta_0) +
 \sum_{j \in \cJ} \langle \beta, u_j \rangle\, \Psi(Z_j)
\right) =
\left( \sum_{j \in \cJ} \langle \alpha, u_j \rangle\, \Psi(Z_j)  \right)\exp(t \, \Psi(\beta_0).
\]
The adjoint representation $\mathrm{Ad}:\bar \G \to \mathrm{Aut}(\bar \g)$ also translates as expected through the map $\Psi$, so that it can be used to apply changes of variables in operator differential equations of the form (\ref{eq:odeXp}).

\subsection{Application: modified equations for splitting methods}
\label{ss:splitting}

 The material just presented may be applied to analyze numerical integrators of differential equations. We refer to \cite{words} for a detailed study of the application of splitting integrators to the solution of perturbations of integrable problem; that study is based on the use of the shuffle Hopf algebra/extended word series.
 Here we show how to proceed when the word series scenario is replaced by the more general framework developed in this section. For simplicity the attention is restricted to the well-known Strang splitting formula.

 Assume that $\tilde F(t)$ is independent of $t$, and that $\exp(t \,\tilde F)$ exists. Then, it is well known that the solution operator $X(t) = \exp(t \, (F_0 +\tilde F))$ can be approximated at $t=\tau, 2\tau,3\tau,\ldots$($\tau$ is the time step)
by $X(k\, \tau) \approx X_k \in \autCC$, where $X_0=I$ and
\begin{equation}
\label{eq:Strang}
X_k = X_{k-1} \exp(\frac{\tau}{2}\, F_0)   \exp(\tau\, \tilde F)  \exp(\frac{\tau}{2}\, F_0), \quad k=1,2,3,\ldots
\end{equation}
If $F_0=\Psi(\beta_0 )$, $\tilde F = \Psi(\tilde \beta)$ with $\beta_0 \in \g_0$ and $\tilde \beta \in \tg$, then
\[
X_{k} = X_{k-1} \left(\sum_{j \in \cJ} \langle\alpha^{\tau}, u_j \rangle\, \Psi(Z_j),
\right),
\]
where
\begin{eqnarray*}
\alpha^{\tau} &=& \textstyle  \exp(\frac{\tau}{2}\, \beta_0) \circ \exp(\tau\, \tilde \beta) \circ \exp(\frac{\tau}{2}\, \beta_0) = \exp(\tau\, \hat \beta^{\tau} ) \circ \exp(\tau\, \beta_0),\\
\hat \beta^{\tau}  & =& \textstyle  \exp(\frac{\tau}{2}\, \beta_0) \cdot \tilde \beta.
\end{eqnarray*}

Let us assume that, for each $n \geq 1$, the restriction to $\g_n$ of
$\int_0^{\tau} \mathrm{Ad}_{\exp(t\, \beta_0)} dt$ is invertible.\footnote{
In some cases, this assumption holds for most values of $\tau \in \R$, but fails for some particular values which gives ris to so-called numerical resonances \cite{words}.}
 In that case,  there exists $\beta^{\tau} \in \g$ such that $\alpha^{\tau} = \exp(\tau \, (\beta_0 + \beta^{\tau}))$, which back to operators, implies that $X_k \in \autCC$ formally coincides with $X^{\tau}(k \tau)$, where $X^{\tau}(t)$ is the formal solution with $X^{\tau}(0)=I$ of the {\em modified equation}
\[
\frac{d}{dt} X^{\tau}(t) = X^{\tau}(t) \left( F_0 + \sum_{i \in \cI} \langle\beta^{\tau}, v_i \rangle\, \Psi(G_i) \right).
\]
Modified equations are of course a powerful tool to analyze the performance of numerical integrators, see e.g.\ \cite{ssc}.

\subsection{Hopf algebraic framework}

To conclude the paper we shall  briefly show how to cast the product group and product Lie algebra constructed in Section
\ref{sec:productgroup} as the group of characters and Lie algebra of infinitesimal characters of  a suitable Hopf algebra.
As in Section \ref{ss:Delta} we consider a (graded) subspace $V$  of
the commutative graded connected Hopf algebra $\H=\bigoplus_{n \geq 0} \H_n$ associated with the graded Lie algebra $\tg =\bigoplus_{n\geq 1} \g_n$. Recall that $V$ has a Lie coalgebra structure such that the Lie algebra $\g$ of infinitesimal characters of $\H$ is isomorphic to the Lie algebra $V^*$ dual to the Lie coalgebra $V$.

In addition to the hypotheses in Section~\ref{sec:productgroup}, we assume that:
\begin{itemize}
\item $\G_0$ is an affine algebraic group with Lie algebra $\g_0$. That is, $\G_0$ (respectively $\g_0$) is the group of  characters (respectively Lie algebra of infinitesimal characters) of a finitely generated commutative Hopf algebra $\bar \H_0$.
\item The action $\cdot$ of the group $\G_0$ on $\g$ can be obtained by dualizing a comodule map $\hat \Delta_0:V \to \bar \H_0 \otimes V$. That is, given $\alpha_0 \in \G_0$,  $\beta \in \g$,
\[
\langle \alpha_0 \cdot \beta, u \rangle = \langle \alpha_0 \otimes \beta, \hat \Delta_0(u) \rangle,
\]
for each $u\in\H$.
\end{itemize}

Then, $\bar \H := \bar \H_0 \oplus \H_1 \oplus \H_2 \oplus \cdots$ can be endowed with a commutative graded Hopf algebra structure in such a way that the resulting group of characters is the semidirect product group $\bar \G$, and the resulting Lie algebra of infinitesimal characters is the semidirect sum Lie algebra
$\bar \g$.

\begin{acknowledgement}
A. Murua and J.M. Sanz-Serna have been supported by projects MTM2013-46553-C3-2-P and MTM2013-46553-C3-1-P from Ministerio de Econom\'{\i}a y Comercio, and MTM2016-77660-P(AEI/FEDER, UE)
from Ministerio de Eco\-nom\'{\i}a, Industria  y Competitividad, Spain.
 Additionally A. Murua has been partially supported by the Basque Government (Consolidated Research Group IT649-13).
\end{acknowledgement}
%
\input{referenc}

%
\end{document}

%% file: referenc.tex
%%%%%%%%%%%%%%%%%%%%%%%% referenc.tex %%%%%%%%%%%%%%%%%%%%%%%%%%%%%%
% sample references
% %
% Use this file as a template for your own input.
%
%%%%%%%%%%%%%%%%%%%%%%%% Springer-Verlag %%%%%%%%%%%%%%%%%%%%%%%%%%
%
% BibTeX users please use
% \bibliographystyle{}
% \bibliography{}
%